\documentclass[a4paper]{article}
\usepackage{amsmath}
\usepackage{amssymb}
\usepackage{amsthm}

\newtheorem{thm}{Theorem}[section]

\newtheorem{lem}[thm]{Lemma}
\newtheorem{prop}[thm]{Proposition}
\theoremstyle{definition}
\newtheorem{defn}[thm]{Definition}
\theoremstyle{remark}
\newtheorem{rem}[thm]{Remark}
\numberwithin{equation}{section}
\newcommand{\norm}[1]{\left\Vert#1\right\Vert}
\newcommand{\abs}[1]{\left\vert#1\right\vert}

\newcommand{\Real}{\mathbb R}
\newcommand{\eps}{\varepsilon}

\newcommand{\para}[1]{\left(#1\right)}
\newcommand{\inp}[1]{\left\langle#1\right\rangle}

\title{ Large Deviation Principle for Mild Solutions of Stochastic Evolution Equations with  Multiplicative L\'{e}vy Noise}%
\author{{Hassan Dadashi}}
\begin{document}

\date{}%

\maketitle
\begin{center}
\emph{Department of Mathematics, IASBS, Zanjan, Iran}
\end{center}
\let\thefootnote\relax\footnotetext{
\hspace{-.5cm}\textbf{Subject Class:}{ 60F10, 60H20}}

\begin{abstract}
We demonstrate the large deviation principle in the small noise limit for the mild solution of stochastic evolution equations with monotone nonlinearity. A recently developed method, weak convergent method,  has been employed in studying the large deviations. we have used essentially  the main result of \cite{kn:BCD} which discloses the variational representation of exponential integrals w.r.t. the  L\'{e}vy noise.  An It\^{o}-type inequality is a main tool in our proofs. Our framework covers a wide range of semilinear parabolic, hyperbolic and delay differential equations. We give some examples to illustrate the applications of the results.
\end{abstract}
\section{\textbf{Introduction}}
\begin{table}[b]
\begin{tabular}{l}
 \hline

\end{tabular}
\end{table}
Partial differential equations are the main framework in modelling the dynamics in physics and many other fields. In particular, to account the undetermined dynamics, stochastic partial differential equations (SPDEs) have been developed in recent decades.  
In studying SPDEs, there are three main approaches: the semigroup, the variational and the martingale measure approach. For the comprehensive references to these approaches, see \cite{kn:DPZ}, \cite{kn:Wa} and \cite{kn:KR}, respectively. In studying some equations like wave equation, it is more natural to consider it in the semigroup approach. Actually one advantages of semigroup approach, which is the framework taken into account in this paper, is that it gives a unified treatment for a wide class of semilinear parabolic, hyperbolic and functional differential equations. 
This approach has been extensively developed in recent decades. See \cite{kn:DPZ} and the references therein for a comprehensive review of progress in this approach. In this paper we investigate an asymptotic property of solution named large deviation principle (LPD).

LDP was initially studied by H. Cram\'{e}r \cite{kn:C} for a sequence of discrete random variables when he was estimating the ruin probability of an insurance company. Then, this notion was developed in more abstract setting. A unified framework in defining large deviation theory for a sequence of measures was developed by S. R. S. Varadhan \cite{kn:Va}, given as follows.
\begin{defn}
Let $(\Omega, \mathcal{F}, \textbf{P})$ be a probability space. The family of random variables $\{X^\eps, \eps>0\}$ on a Polish space (a separable complete metric space) $\mathcal{E}$ is said to satisfy the LDP with the good \emph{rate function} (or \emph{action functional}) $I$, where $\emph{I}:\mathcal{E}\rightarrow [0,\infty]$ is a lower semicontinuous function, if\\

$A1)$ for every $N< \infty$, the level set
$$K_N:=\{x\in \mathcal{E} : \emph{I}(x)\leq N\},$$ 
is compact.\\

$A2)$ for any open set $G\in \mathcal{E}$, the following upper bound satisfies
\begin{align}
\limsup\limits_{\eps\rightarrow0} \eps^2 \log \textbf{P}(X^{\eps}\in G)\leq -\inf\limits_{x\in G}\emph{I}(x);\nonumber
\end{align}

$A3)$ for any closed set $F\in \mathcal{E}$, the following lower bound satisfies
\begin{align}
\liminf\limits_{\eps\rightarrow0} \eps^2 \log \textbf{P}(X^{\eps}\in F)\geq -\inf\limits_{x\in F}\emph{I}(x).\nonumber
\end{align}
\end{defn}

The theory of large deviations, which has been investigated for different systems in recent years, reveals important aspects of asymptotic dynamics. Principally it gives the probability of deviation from an equilibrium point. See \cite{kn:Hol} for a basic and \cite{kn:DZ, kn:FW, kn:OV} for a  comprehensive introductions to this topic. In recent years, particular attention has been paid to studying LDP for stochastic differential equations (cf. e.g. \cite{kn:Az, kn:Fr, kn:So, kn:Pe,  kn:RZZ}). 

There exist two distinct methods, the classical and the weak convergence method, in establishing   LDP for an SDE or SPDE with multiplicative noise. In the classical method  we should discretize the time horizon and freez the diffusion term on each interval and then use the Varadhan's contraction principle. In this method we should overcome many difficult inequalities for convolution integrals. In the weak convergence approach, which is the approach employed in this work, we should obtain some sort of continuity w.r.t. some control variables. We will clarify further this approach in the current and next section.  

Several authors have studied LDP for  infinite dimensional SDEs with L\'{e}vy noise. See \cite{kn:CR} for the classical and \cite{kn:BCD} for the weak convergence approach. 

Varadhan's and Bryc's results, \cite{kn:Va} and \cite{kn:B}, announced an equivalence between LDP and Laplace principle (LP), which notices the expectations of exponential functions. 

\begin{defn} (\textbf{Laplace principle}) The family of random variables $\{X^\eps\}$, defined on the Polish space $\mathcal{E}$, is said to satisfy the \emph{Laplace principle} with the \emph{rate function I} if for all bounded continuous functions $h:\mathcal{E}\rightarrow\Real$, 
$$\lim_{\eps\rightarrow0}\eps^2 \log \mathbb{E}\;{\exp\left\{-\frac{1}{\eps^2}h(X^\eps)\right\}}=-\inf\limits_{f\in \mathcal{E}}\{h(f)+\emph{I}(f)\}.$$
\end{defn}
Another display of variational representation in evaluating the exponential integrals is in the following proposition which is a cornerstone of weak convergence method. See \cite{kn:DE} for a comprehensive introduction to applications of weak convergence method to the theory of large deviations. 
\begin{prop}
Let $(\mathcal{V}, \mathcal{A})$ be a measurable space and $f$ be a bounded measurable function mapping V into the real numbers $\mathbb{R}$. For a given probability measure $\theta$ on $\mathcal{V}$, we have the following representation 
$$-\log \int_{\mathcal{V}}e^{-f}d\theta=\inf_{\gamma\in \mathcal{P}(\mathcal{V})}\{\mathcal{R}(\gamma||\theta)+\int_{\mathcal{V}}fd\gamma\},$$
where $\mathcal{R}(\gamma||\theta):=\int_{\Omega}\log \frac{d\gamma}{d\theta}(\omega)\gamma(\omega)$. 
\end{prop}
By using the above representation, Budhiraja and Dupuis \cite{kn:BD} obtained the following variational representation for the exponential integrals w.r.t. the Wiener process.  
\begin{align}
-\log \mathbb{E}\;\{\exp(-h(W))\}=\inf \limits_{u\in\mathcal{P}^2(U)}\mathbb{E}\;\para{\frac{1}{2}\displaystyle\int^T_0\norm{u(s)}^2_U ds+h\para{W+\displaystyle\int^._0u(s)ds}}\nonumber,
\end{align}
in which, $U$ is a Hilbert space, $h:\mathcal{C}([0,T];U)\rightarrow \Real$ is a bounded and Borel measurable function, $W(t)$ is a cylindrical Brownian motion and $\mathcal{P}^2(U)$ is the family of all predictable processes on $U$. This representation together with Laplace principle gives a different method in obtaining LDP for SDEs with Gaussian noise. The corresponding representation for L\'{e}vy noise, obtained in \cite{kn:BCD}, is being expressed  in the next section.

The aim of this work is to apply the weak convergence approach to establish LDP for the mild solution (Definition \ref{mildsol}) of equation
\begin{align}
dX_t=[AX_t+f(X_t)]dt+\int_{\mathbf{X}} G(X_t)\tilde{N}(dvdt),\label{mainlevy}
\end{align}
where $A$ is a linear differential operator, $f$ is a monotone demicontinuous function and $\tilde{N}$ is a compensated Poisson random measure with state space $\mathbf{X}$. This equation gives a suitable setting in studying many dynamical systems. We give some examples in the last section.

In the next section we give some preliminaries and assumptions. In the third section, we show the main results and establish the LDP in three subsections. Finally, in the last section we give some examples and applications.



\section{\textbf{Preliminaries and Main Result}}

\subsection{Poisson Random Measure and a Variational Representation}

We choose a locally compact Polish space $\mathbf{X}$ as the state space of jumps. In practice, a subset of $\mathbb{R}^n$ can be chosen as $\mathbf{X}$. Let $\mathcal{M}_{FC}(\mathbf{X})$ be the space of all measures $\nu$ on $(\mathbf{X}, \mathcal{B}(\mathbf{X}))$ such that $\nu(K)<\infty$ for every compact subset $K\subset \mathbf{X}$. Consider the weakest topology on $\mathcal{M}_{FC}(\mathbf{X})$ such that for every $f\in \mathcal{C}_c(\mathbf{X})$ (the space of all continuous function with compact support), the function $\nu\rightarrow \inp{f,\nu}=\int_{\mathbf{X}}f(u)d\nu(u)$ is continuous on $\mathcal{M}_{FC}(\mathbf{X})$. To consider the time horizon let $\mathbf{X}_T=[0,T]\times\mathbf{X}$ and for a measure $\nu\in\mathcal{M}_{FC}(\mathbf{X})$ let $\nu_T=\lambda_T\times\nu$, where $\lambda_T$ is the Lebesgue measure on $[0,T]$. 

\begin{defn}\textbf{(Poisson Random Measure)}

A Poisson random measure $N$ on the space $\mathbf{X}_T$ with intensity measure $\nu_T$ is an $\mathcal{M}_{FC}(\mathbf{X}_T)$-valued random variable such that for every $B\in\mathcal{B}(\mathbf{X}_T)$ with $\nu_T(B)<\infty$, $N(B)$ is Poisson random variable with intensity $\nu_T(B)$ and for disjoint $B_1,\cdots, B_k\in \mathcal{B}(\mathbf{X}_T)$, the random variables $N(B_1), \cdots, N(B_k)$ are mutually independent.
\end{defn}

Let $\mathbb{P}$ be the measure induced by $N$ on $\mathbb{M}:=\mathcal{M}_{FC}(\mathbf{X}_T)$. We recall that $\mathbb{P}$ is the unique probability measure on this space such that the canonical map $N:\mathbb{M}\rightarrow \mathbb{M}$, $N(m):=m$ is a Poisson random measure with intensity measure $\nu_T$.

Now, we give some preliminaries to state the variational representation of exponential integrals w.r.t. Poisson random measure $N$. In weak convergence method, control functions play a crucial role in establishing the variational representations. In the Wiener noise case, the absolutely continuous functions have been chosen as the control functions, see \cite{kn:BD}. Clearly, in the L\'{e}vy noise case, this is not an appropriate choice and it seems that a variation on the intensity measure should be possible by the control functions. 

To change the jump intensity at time $t\geq 0$ and jump vector $x\in \mathbf{X}$, we multiply the intensity measure $\nu_T$ over the space $\mathbf{X}_T$ by a nonnegative function $g:\mathbf{X}_T\rightarrow [0,\infty)$. To express the modified measure as a new Poisson random measure, we can consider the space $\mathbf{Y}=\mathbf{X}\times [0,\infty)$  as new state space.  Then letting $\mathbf{Y}_T=[0,T]\times \mathbf{Y}$ and $\bar{\mathbb{M}}=\mathcal{M}_{FC}(\mathbf{Y}_T)$, there is a unique probability measure $\bar{\mathbb{P}}$ on the space $(\bar{\mathbb{M}}, \mathcal{B}(\bar{\mathbb{M}}))$ under which the canonical map, $\bar{N}:\bar{\mathbb{M}}\rightarrow \bar{\mathbb{M}}$, $\bar{N}(m):=m$ is a Poisson random measure with intensity measure $\bar{\nu}_T=\lambda_T\otimes \nu\otimes\lambda_{\infty}$, where $\lambda_{\infty}$ is the Lebesgue measure on $[0,\infty)$. The corresponding expectation operator will be denoted by $\bar{\mathbb{E}}$. 

Let $\mathcal{F}_t:=\sigma\{\bar{N}((0,t]\times A):0\leq s\leq t, A\in\mathcal{B}(\mathbf{Y})\}$ and $\bar{\mathcal{F}}_t$ be its completion under $\bar{\mathbb{P}}$. We denote by $\bar{\mathcal{P}}$ the predictable $\sigma$-field on $[0,T]\times\bar{\mathbb{M}}$ with the filtration $\{\bar{\mathcal{F}}_t, 0\leq t\leq T\}$. We consider all $(\mathcal{B}(\mathbf{X})\otimes\bar{\mathcal{P}})/\mathcal{B}[0,\infty)$-measurable maps $g:\mathbf{X}_T\times\bar{\mathbb{M}}\rightarrow[0,\infty)$ as control functions and denote this class by $\bar{\mathcal{A}}$. Therefore assuming $(\bar{\mathbb{M}},\mathcal{B}(\bar{\mathbb{M}}))$ as the probability space, a control function alters randomly the intensity at the point $(t,x)\in \mathbf{X}_T$, in a predictable way. 

For an arbitrary $g\in\bar{\mathcal{A}}$, we define the new counting process $N^{g}$ on the space $\mathbf{X}_T$ by 
\begin{align}
N^{g}((0,t)\times U)=\int_{(0,t]\times U}\int_{(0,\infty)}1_{[0,g(s,x)]}(r)\bar{N}(dsdxdr),\qquad t\in[0,T], U\in \mathcal{B}(\mathbf{X}).\nonumber
\end{align}
So, $N^{g}$ is a modified version of $N$ in which $g$ alters the intensity at the point $x$ and time $t$. When $g(a,x,\bar{m})=\theta$ is constant, we write $N^{g}=N^{\theta}$ and the corresponding expectation operator will be denoted by $\mathbb{E}^{\theta}$. 

Define $l:[0,\infty)\rightarrow [0,\infty)$ by
$$l(r)=r\log r-r+1, \qquad r\in [0,\infty).$$
We consider for any $g\in\bar{\mathcal{A}}$ the quantity 
$$L_T(g)=\int_{\mathbf{X}_T}l(g(t,x,\omega))\nu_T(dtdx),$$
as a $[0,\infty]$-valued random variable. In the Wiener noise model, the corresponding quantity  is  defined, see in \cite{kn:BD}, by the energy of the absolutely continuous control functions.

The following representation is the essential tool in establishing LDP in the L\'{e}vy noise case and was proved in \cite{kn:BDM}.

\begin{thm}
Let $F\in \mathcal{M}_b(\mathbb{M})$ be a bounded measurable function on $\mathbb{M}$. Then
\begin{align}
-\log\mathbb{E}^{\theta}(e^{-F(N)})=-\log\bar{\mathbb{E}}(e^{-F(N^{\theta})})=\inf_{g\in\bar{\mathcal{A}}}\bar{\mathbb{E}}[\theta
 L_T(g)+F(N^{\theta g})].\label{levyrep}
\end{align}

\end{thm}

\subsection{A General Large Deviation Principle}

Now, we write the required condition for a general family of random variables $\{X^{\eps}=\mathcal{G}^{\eps}(\eps N^{\eps^{-1}}), \eps>0\}$ to satisfy LDP, where  $\{\mathcal{G^{\eps}}, \eps>0\}$ is a family of measurable maps from $\mathbb{M}$ to a Polish space $\mathbb{U}$. Consider the set 
$$S^N=\{g:\mathbf{X}_T\rightarrow [0,\infty) :L_T(g)\leq N \}.$$ 
We can identify an arbitrary $g\in S^N$ by a measure $\nu^g_T\in \mathbb{M}$, defined by 
$$\nu^g_T(A)=\int_Ag(s,x)\nu_T(dsdx), \qquad A\in\mathcal{B}(\mathbb{X}_T).$$
This identification induces a topology on $S^N$, under which $S^N$ is a compact space. See \cite[Section 5.1]{kn:BDM} for a proof of this statement. We use this topology on $S^N$ in the sequel. Define $\mathbb{S}=\cup _{N\geq 1}S^N$ and let 
$$\mathcal{U}^N=\{g\in \bar{\mathcal{A}}: g(\omega)\in S^N  \bar{\mathbb{P}}\; a.e. \}. $$
The following conditions will be sufficient to establish LDP for the family $\{X^{\eps}=\mathcal{G}^{\eps}(\eps N^{\eps^{-1}}), \eps>0\}$. \\

\textbf{Hypothesis 1} 
There are measurable maps $\mathcal{G}^0, \mathcal{G}^{\eps}:\mathbb{M}\rightarrow \mathbb{U}$  that  satisfy the following conditions. 

\begin{description}
\item (i) Let $g^n, g\in S^N$ be such that $g^n\rightarrow g$ as $n\rightarrow \infty$. Then
$$\mathcal{G}^0(\nu^{g^n}_T)\rightarrow \mathcal{G}^0(\nu^g_T).$$
\item (ii) Let $g^{\eps}, g\in \mathcal{U}^N$ be such that $g^{\eps}$ converges in distribution (as random variables) to $g$. Then 
$$\mathcal{G}^{\eps}(\eps N^{\eps^{-1}}g^{\eps})\rightarrow \mathcal{G}^0(\nu^g_T).$$

\end{description}

The first condition represents the continuity of deterministic system w.r.t. the control functions. The second condition is a law of large numbers, when the intensity of control functions tends to 
zero. 

To determine the rate function, let $\mathbb{S}_{\phi}=\{g\in\mathbb{S}: \phi=\mathcal{G}^0(\nu^g_T)\}$. Define the rate function $I:\mathbb{U}\rightarrow [0,\infty]$ by 
\begin{align}
I(\phi)=\inf_{g\in\mathbb{S}_{\phi}}\{L_T(g)\}, \qquad \phi\in \mathbb{U},\label{rate}
\end{align}

and by convention let $I(\phi)=\infty$ if $\mathbb{S}_{\phi}=\varnothing$.\\

The following theorem, obtained in \cite{kn:BDM}, is our main reference in establishing LDP. 

\begin{thm}\label{mainthm}
For $\eps>0$ let $X^{\eps}$ be defined by $X^{\eps}=\mathcal{G}^{\eps}(\eps N^{\eps^{-1}})$, and suppose that Hypothesis 1 holds. Then the function $I$, defined as in (\ref{rate}), is a rate function on $\mathbb{U}$ and the family $\{X^{\eps}, \eps>0\}$ satisfies a large deviation principle on $\mathbb{U}$ with rate function $I$. 
\end{thm}
\subsection{A Class of SPDE}

In studying SPDEs, the semigroup approach is been considered as an interesting approach which involves a wide range of semilinear SPDEs. In this approach, it is assumed that the solution at any time lies in a Hilbert space $H$. The typical semilinear  equation with Gaussian noise has the following form 
 
\begin{align}
dX_t=[AX_t+F(X_t)]dt+G(X_t)dW_t,\label{semigroupeq}
\end{align}
in which $A$ is a (possibly unbounded) differential operator that generates a strongly continuous semigroup $S(t), t\geq 0$ on $H$. In this paper, we consider a semilinear equation with L\'{e}vy noise, which takes the following form. 
\begin{align}
dX_t=[AX_t+f(X_t)]+\int_{\mathbf{X}}G(t,X_t,v)\tilde{N}(dvdt).\label{mainlevyeq}
\end{align}

We explore the mild solution of above equation over the interval $[0,T]$ in the space of $H$-valued right continuous with left limit functions (RCLL), which is denoted by $\mathcal{D}([0,T];H)$. We consider the uniform metric $\norm{f-g}_{\infty}=\sup_{0\leq t\leq T}\norm{f(t)-g(t)}$ on this space which makes it a complete space. Here $\norm{\cdot}$ denotes the norm in the space $H$. Although the space $\mathcal{D}([0,T];H)$ is usually endowed with the Skorokhod topology, which is a weaker topology, the solution in our setting  can be established in the uniform topology. 

\begin{defn}(\textbf{Mild Solution})\label{mildsol}
For the filtered probability space\\ 
$(\bar{\mathbb{M}}, \mathcal{B}(\bar{\mathbb{M}}), \bar{\mathbb{P}}, \{\bar{\mathcal{F}}\})$, introduced in subsection 2.1, let $X_t$ be a predictable process  with $\bar{\mathbb{P}}$ a.s. paths in the space $\mathcal{D}([0,T];H)$. Furthermore, let $X_0\in H$ be a $\bar{\mathcal{F}}_0$-measurable square integrable random variable, $\mathbb{E}\norm{X_0}^2<\infty$.  $X_t$ is the mild solution of Equation (\ref{mainlevyeq}) with initial value $X_0$, if it is the strong solution of the following integral equation
\begin{align}
X^\eps_t=S(t)X_0+\int^t_0S(t-s)f(X^\eps_s)ds+\int^t_0\int_{\mathbf{X}}S(t-s) G(s,X^\eps_s,v)\tilde{N}(dvds).\label{mildsol}
\end{align}
\end{defn}

The imposed assumptions on the linear operator $A$, the nonlinear drift coefficient $f:H\rightarrow H$ and the diffusion coefficient $G:[0,T]\times H\times \mathbf{X}\rightarrow H$ are as the followings. \\

\textbf{Hypothesis 2}
\begin{description}

  \item (i) $A:D(A)\rightarrow H$ is a closed linear operator with dense domain $D(A)\subseteq H$ which generates a $C_0$-semigroup $S(t)$ on $H$ and there exist $L, \lambda >0$ such that $\norm{S(t)}\leq L e^{\lambda t}$, for every $t\geq 0$;\\

  \item (ii) $f$ is a demicontinuous function on $H$, i.e., whenever $\{x_n\}$ converges strongly to $x$ in $H$, $\{f(x_n)\}$ converges weakly to $f(x)$ in $H$;\\

  \item (iii) $f$ satisfies a linear growth condition, i.e.
$$\exists C>0; s.t. \quad \norm{f(x)} \leq C(1+\norm{x}),\qquad \forall x\in H;$$

  \item (iv) $-f$ is semimonotone with parameter $M\geq 0$, i.e.
$$\inp{f(x)-f(y),x-y}\leq M\norm{x-y}^2 ,\qquad \forall x,y\in H;$$

  \item (v) The map $x\rightarrow G(t,x,\cdot)$ is uniformly Lipschitz continuous with Lipschitz constant $M$, i.e.
$$\int_{\mathbf{X}}\norm{G(t,x,v)-G(t,y,v)}^2\nu(dv) \leq M \norm{x-y}^2 ,\qquad \forall x,y\in H.$$
\end{description}

\begin{rem} Without loss of generality, we can assume in Hypothesis 2(i) that $\lambda=0$ and $L=1$. This means that $S(t)$ is a contraction. This simplifying assumption has been argued in \cite{kn:Z} in proving the existence of solution. 
\end{rem}

To obtain the large deviation result, we need some more assumptions on the diffusion coefficient.  Consider the following two norms on $G$: \\

$$\norm{G(t,v)}_0=\sup_{x\in H}\frac{\norm{G(t,x,v)}}{1+\norm{x}},$$

$$\norm{G(t,v)}_1=\sup_{x, y\in H}\frac{\norm{G(t,x,v)-G(t,y,v)}}{\norm{x-y}}.$$

\textbf{Hypothesis 3}
There exists $\delta >0$ such that for all Borel sets $E\in\mathcal{B}([0,T]\times \mathbf{X})$ satisfying $\nu_T(E)<\infty$, the followings hold 
\begin{align}
\int_{E}e^{\delta \norm{G(t,v)}^2_0}\nu(dv)dt<\infty, \label{expcon1}
\end{align}
\begin{align}
 \int_{E}e^{\delta \norm{G(t,v)}^2_1}\nu(dv)dt<\infty. \label{expcon2}
\end{align}

\begin{rem}
Under Hypothesis 3, the Estimates (\ref{expcon1}), (\ref{expcon2}) hold for every $\delta >0$ and all $E\in\mathcal{B}([0,T]\times \mathbf{X})$ satisfying $\nu_T(E)<\infty$.
\end{rem}

Many authors have been studied the existence, uniqueness and qualitative properties of mild solutions for different assumptions on the nonlinear coefficient $f$. By assuming monotonicity on $f$, which is weaker than the Lipschitz condition, we can treat a wider class of dynamics, see \cite{kn:FJ} for an important application. In \cite{kn:Z2, kn:Za}, Zangeneh  explored the mild solution of a more general semilinear stochastic evolution equation with monotone nonlinearity. He proved the existence, uniqueness, and some qualitative properties of the solution of the integral equation

\begin{align}
X_t=U(t,0)X_0+\int^t_0 U(t,s)f(s,X_s)ds+\int^t_0 U(t,s)g(s,X_s)dW_s+V(t),\nonumber
\end{align}
in which $V(t)$ is an arbitrary Gaussian process and the differential operator $A(t)$ is time dependent and therefore generates a two parameter semigroup $U(t,s), 0\leq s\leq t$. Some other properties of solution, including LDP and asymptotic stability, have been studied in \cite{kn:DZ, kn:JaZ, kn:ZZ}. For the L\'{e}vy noise case, the existence and uniqueness of mild solution of Equation (\ref{mainlevyeq}) with assumptions similar to Hypothesis 2 have been studied recently in \cite{kn:SZ}.  


In the following we give two inequalities, proved respectively in \cite{kn:Zan} and \cite{kn:Za}, which give upper bounds for the norm of convolution integrals and are being used in our arguments several times. The first inequality  considers the deterministic equations and the second one gives the corresponding inequality for stochastic equations.

\begin{prop}(\textbf{Energy-type inequality})
Let $a:[0,T]\rightarrow H$ be an integrable function. Suppose $A$ and $S$ satisfy Hypothesis 2(i). If
\begin{align}
X_t=S(t)X_0+\displaystyle\int^t_0S(t-s)a(s)ds,\nonumber
\end{align}
then
\begin{align}
\norm{X_t}^2\leq e^{2\lambda t}\norm{X_0}^2+2\displaystyle\int^t_0 e^{2\lambda(t-s)}\inp{X(s),a(s)}ds,\quad\forall t\in[0,T].\nonumber
\end{align}
\end{prop}

\begin{prop}(\textbf{It\^{o}-type inequality})
Let $\{Z_t: \; t\in[0,T]\}$ be an $H$-valued, cadlag, locally square integrable semimartingale. Suppose $A$ and $S$ satisfy Hypothesis 2(i). If
\begin{align}
X_t=S(t)X_0+\displaystyle\int^t_0S(t-s)dZ_s.\nonumber
\end{align}
Then for all $t\in[0,T]$,
\begin{align}
\norm{X_t}^2\leq e^{2\lambda t}\norm{X_0}^2+2\displaystyle\int^t_0 e^{2\lambda(t-s)}\inp{X(s),dZ_s}+e^{2\lambda t}\left[\displaystyle\int^._0 e^{-\lambda s}dZ_s\right]_t,~~~w.p. 1,\nonumber
\end{align}
where $[ \quad ]_t$ stands for the quadratic variation process.
\end{prop}

A consequence of Burkholder's inequality together with the Young inequality 
$$ab\leq K\frac{a^p}{p}+K^{-\frac{q}{p}}\frac{b^q}{q},\qquad a, b, K>0;\qquad \frac{1}{p}+\frac{1}{q}=1,$$ 
is the following lemma, which is being used in Subsection 3.3.
\begin{lem}\label{lem1} Let $X_t, \; t\in [0,\infty)$, be an $H$-valued continuous process. if $M_t$ is an $H$-valued square integrable cadlag martingale, then for any constant $K>0$, we have
$$\mathbb{E}\left\{\sup\limits_{0\leq \theta\leq t}\abs{\displaystyle\int^{\theta}_0\inp{X_s,dM_s}}\right\}\leq\frac{3}{2K}\mathbb{E}(X^*_t)^2+\frac{3K}{2}\mathbb{E}([M]_t),$$
in which $X^{*}_t=\sup\limits_{0\leq s\leq t}\norm{X_s}$.
\end{lem}

To obtain a uniform estimate of stochastic convolutions, we will apply in Subsection 3.3 the following proposition.
\begin{prop}(\textbf{Burkholder-type inequality})
Suppose $A$ and $S$ satisfy Hypothesis 2(i) and $M$ is an $H$-valued square integrable cadlag  martingale. If $p\geq1$, then there exist suitable constant $C$ such that
\begin{align}
\mathbb{E}\left\{\sup\limits_{0\leq t\leq T}\norm{\displaystyle\int^t_0S(t-s)dM_s}^{2p}\right\}\leq Ce^{\lambda T}\mathbb{E}\{[M]^p_T\}, \quad \forall T>0.\label{18}
\end{align}\
\end{prop}
An extension of this proposition for stopping times has been proved in \cite{kn:HZ1}.

The well known lemma, Gronwall's lemma, are being used  in our arguments several times. 
\begin{lem}(\textbf{Gronwall's lemma})
Let the functions $f:\Real\rightarrow\Real$ and $g:\Real\rightarrow [0,\infty)$ satisfy 
$$f(s)\leq a+\int^s_0f(r)g(r)dr $$
for some $t\geq 0$ and all $s\in[0,t]$. Then we have the estimate $f(t)\leq ae^{\int^t_0g(s)ds}.$
\end{lem}

\subsection{The Main Result}

To clarify the weak convergence framework in our setting, we define for any $g\in S^N$ the map  $\mathcal{G}^0(\nu^g_T)$, as the strong solution of the  integral equation
\begin{align}
X^g_t=S(t)X_0&+\int^t_0S(t-s)f(X^g_s)ds\nonumber\\
&+\int^t_0S(t-s)\int_{\mathbf{X}}G(s,X^g_s,v)(g(s,v)-1)
\nu(dv)ds,\label{skeleton}
\end{align}
where $S(t), f$ and $G$ are being specified in Hypothesis 2 and $X_0$ satisfies $\mathbb{E}\norm{X_0}^2<\infty$. Furthermore, for any $g^{\eps}\in S^N$, $\mathcal{G}^{\eps}(\eps N^{\eps^{-1} g^{\eps}})=X^{\eps}$ is the mild solution of the following  evolution equation with initial condition $X^{\eps}(0)=X_0$,
\begin{align}
dX^{\eps}_t=[A X_t+f(X^{\eps}_t)]dt+\int_{\mathbf{X}}G(s,X^{\eps}_{t-},v)\left(\eps N^{\eps^{-1} g^{\eps}}(dvdt)-\nu(dv)dt\right).\label{maineps}
\end{align}

The following theorem is the main result we are pursuing in this paper.

\begin{thm}\label{mainthm}
Let the Hypotheses 2 and 3 be satisfied and $X^{\eps}_t$ be the mild solution of the  evolution equation
\begin{align}
dX^{\eps}_t=[AX^{\eps}_t+f(X^{\eps}_t)]+\eps\int_{\mathbf{X}}G(t,X^{\eps}_t,v)\tilde{N}^{\eps^{-1}}(dvdt).\label{ldpeq}
\end{align}
Then the family of measures generated by the random variables $\{X^{\eps}, \eps>0\}$ on the space $\mathcal{D}([0,T];H)$ satisfies LDP in the uniform topology.
\end{thm}
To prove this theorem, we should verify that the defined $G^0, G^{\eps}$ satisfy Hypothesis 1.


\section{Large Deviation Principle} 
In the first step, we show that  the map $\mathcal{G}^0$ is well defined. In the second subsection we will ascertain that the function defined in (\ref{rate}) is a good rate function. Finally in the third subsection we will verify  Hypothesis 1(ii).

\subsection{\textbf{The Skeleton Problem}}
We must prove that for every measurable $g\in \mathbb{S}$, Equation (\ref{skeleton}) has a  solution. In the proof, we will use the following lemmas, which has been proved in \cite[Lemma 3.4, Lemma 3.11]{kn:BDM}. 

\begin{lem}\label{gint}
Under the Hypothesis 2(v), for every $N\in\mathbb{N}$, the following estimates hold
$$\sup_{g\in S^N}\int_{\mathbb{X}_T}\norm{G(t,v)}^2_0(g(t,v)+1)\nu(dv)dt<\infty,$$

$$\sup_{g\in S^N}\int_{\mathbb{X}_T}\norm{G(t,v)}_0|g(t,v)-1|\nu(dv)dt<\infty.$$
\end{lem}

\begin{rem}\label{gintrem}
The above estimates satisfy if we replace the norm $\norm{\cdot}_0$ by $\norm{\cdot}_1$, see \cite[Remark 3.6]{kn:BDM}.
$$\sup_{g\in S^N}\int_{\mathbb{X}_T}\norm{G(t,v)}^2_1(g(t,v)+1)\nu(dv)dt<\infty,$$

$$\sup_{g\in S^N}\int_{\mathbb{X}_T}\norm{G(t,v)}_1|g(t,v)-1|\nu(dv)dt<\infty.$$
\end{rem}

\begin{lem}\label{hcon}
Fix $N\in\mathbb{N}$ and let $g_n, g\in S^N$ be such that $g_n\rightarrow g$ as $n\rightarrow \infty$. Let $h:[0,T]\times\mathbf{X}\rightarrow\Real$ be a measurable function such that 
$$\int_{\mathbf{X}_T}|h(t,v)|^2\nu_T(dvdt)<\infty.$$
Furthermore suppose that
$$\int_{E}e^{\delta |h(t,v)|}\nu_T(dvdt),$$
for all $\delta\in(0,\infty)$ and all $E\in\mathcal{B}([0,T]\times\mathbf{X})$ satisfying $\nu_T(E)<\infty$. Then we have 
$$\int_{\mathbf{X}_T}h(t,v)(g_n(t,v)-1)\nu_T(dvdt)\rightarrow \int_{\mathbf{X}_T}h(t,v)(g(t,v)-1)\nu_T(dvdt),$$
as $n\rightarrow \infty$.
\end{lem}

\begin{thm}\label{skeletonthm}
Let the assumptions of Hypotheses  2 and 3 be satisfied. Then for every measurable function $g\in \mathbb{S}$, the following equation has the strong solution  $X^g_t$,
\begin{align}
X^g_t=S(t)X_0&+\int^t_0S(t-s)f(X^g_s)ds\nonumber\\
&+\int^t_0S(t-s)\int_{\mathbf{X}}G(s,X^g_s,v)(g(s,v)-1)\nu(dv)ds.\label{maineq}
\end{align}
\end{thm}  

Before giving the proof, we should obtain an a priori bound for the solution $X^g_t$.

\begin{lem}\label{apriorilem}
Let the assumptions of Theorem \ref{skeletonthm} be satisfied and $X^g$ be the solution of (\ref{maineq}). Then the solution is uniformly bounded, 
\begin{align}
\sup_{0\leq t\leq T}\norm{X^g_t}<\infty.\label{apriori}
\end{align}
\end{lem}

\begin{proof}
We use the energy type inequality to get 
\begin{align}
\norm{X^g_t}^2&\leq \norm{X_0}^2+ 2\int^t_0 \inp{X^g_s,f(X^g_s)+\int_{\mathbf{X}}G(s,X^g_s,v)(g(s,v)-1)
\nu(dv)}ds\nonumber\\
&\leq \norm{X_0}^2+ 2\int^t_0(1+C\norm{X^g_s}^2)ds\nonumber\\
&\qquad +2\int^t_0(1+2\norm{X^g_s}^2)\int_{\mathbf{X}}\norm{G(s,v)}_0(g(s,v)-1)
\nu(dv)ds\nonumber
\end{align}
We have used the linear growth condition on $f$ and the elementary inequality $x+x^2\leq 1+2x^2$ in the second inequality. Denoting $M=\int_{\mathbf{X_T}}\norm{G(s,v)}_0(g-1)\nu(dv)ds$ and $g(s)=\int_{\mathbf{X}}\norm{G(s,v)}_0(g-1)\nu(dv)$,  the Gronwall's lemma implies   
\begin{align}
\norm{X^g_t}^2\leq ae^{\int^t_0g(s)ds},\nonumber
\end{align} 
in which $a=\norm{X_0}^2+2t+2M$. Then applying Lemma \ref{gint} in estimating $\int^t_0g(s)ds$  implies the desired estimate. Note that this argument gives  a uniform bound for all $g\in S^N$.
\end{proof}

We should notice that the positive constants $C, M$ in our arguments come from the assumptions and may change from line to line.\\

\textbf{\textit{Proof of Theorem \ref{skeletonthm}.}}
For a bounded function $g:\mathbf{X}_T\rightarrow [0,\infty)$, the Equation (\ref{maineq}) is a  deterministic equation with monotone nonlinearity which has been studied in \cite{kn:Br}. For an arbitrary $g\in \mathbb{S}$, let $g^n=g1_{g^{-1}[-n,n]}$ be the truncated function. Note that since $L_T(g)\leq N$ for some $N<\infty$, there exists a $K>0$ such that for all $n>N$ and outside of a set $A\subset \mathbf{X}_T$ with 
$\mu(A)<\epsilon$, we have  $g^n=g$. 

From the definitions we write 
\begin{align}
X^{g^n}_t-X^{g^m}_t=&\int^t_0S(t-s)(f(X^{g^n}_s)-f(X^{g^m}_s))ds\nonumber\\
&+\int^t_0\int_{\mathbf{X}}S(t-s)[G(s,X^{g^n}_s,v)-G(s,X^{g^m}_s)]g^n(s,v)\nu(dv)ds\nonumber\\
&+\int^t_0\int_{\mathbf{X}}S(t-s)G(s,X^{g^m}_s,v)[g^n(s,v)-g^m(s,v)]
\nu(dv)ds.\nonumber
\end{align}
Due to the energy-type inequality we have 
\begin{align}
&\norm{X^{g^n}_t-X^{g^m}_t}^2\nonumber\\
 &\hspace{1cm}\leq\int^t_0\inp{X^{g^n}_s-X^{g^m}_s,f(X^{g^n}_s)-f(X^{g^m}_s)}ds
\nonumber\\
&\hspace{1.3cm}+\int^t_0\inp{X^{g^n}_s-X^{g^m}_s,\int_{\mathbf{X}}[G(s,X^{g^n}_s,v)-G(s,X^{g^m}_s)]g^n(s,v)\nu(dv)}ds\nonumber\\
&\hspace{1.3cm}+\int^t_0\inp{X^{g^n}_s-X^{g^m}_s,\int_{\mathbf{X}}G(s,X^{g^m}_s,v)[g^n(s,v)-g^m(s,v)]
\nu(dv)}ds\nonumber\\
&\hspace{1cm}\leq I_1(t)+I_2(t)+I_3(t).\nonumber
\end{align}
By using Hypothesis 2(iv) we have 
$$I_1(t)\leq M\int^t_0 \norm{X^{g^n}_s-X^{g^m}_s}^2ds,$$
Moreover, for the term $I_2(t)$ we have 
\begin{align}
I_2(t)\leq & \int^t_0 \norm{X^{g^n}_s-X^{g^m}_s}^2\int_{\mathbf{X}}\norm{G(s,v)}_1g^n(s,v)\nu(dv)ds.\label{I2}
\end{align}
Finally, by the Causchy-Schwartz inequality we get 
\begin{align}
I^2_3(t)&\leq \int^t_0 \norm{X^{g^n}_s-X^{g^m}_s}^2ds \int^t_0\left\{\int_{\mathbf{X}}G(s,X^{g^m}_s,v)[g^n(s,v)-g^m(s,v)]\nu(dv)\right\}^2ds\nonumber\\
&\leq \int^t_0 \norm{X^{g^n}_s-X^{g^m}_s}^2ds \times\nonumber\\
&\qquad\int^t_0(1+\norm{X^{g^m}_s})^2\left\{\int_{\mathbf{X}}\norm{G(s,v)}_0[g^n(s,v)-g^m(s,v)]\nu(dv)\right\}^2ds.\nonumber
\end{align}
Then by using Lemma \ref{hcon} and the a priori estimate, the second integral at the right hand side tends to zero. Therefore, due to the above estimates for $I_1, I_2, I_3$, using Remark in the estimate (\ref{I2}) and applying Gronwall's inequality we conclude that 
$\norm{X^{g^n}-X^{g^m}}_{\infty}$ converges to  zero. This means that $X^{g^n}$ is a Cuachy sequence in the space 
$\mathcal{D}([0,T];H)$ and therefore converges to a limit, named $X^g$. To complete the proof, we have to show that $X^g$ is a solution of (\ref{maineq}). To this end, we must prove the following convergences, as $n\rightarrow \infty$
\begin{align}
\int^t_0S(t-s)f(X^{g^n}_s)ds\rightarrow \int^t_0S(t-s)f(X^g_s)ds,\label{fcon}
\end{align} 

and
\begin{align}
\int^t_0S(t-s)\int_{\mathbf{X}}G(s,X^{g^n}_s,v)(g^n(s,v)-1)\nu(dv)ds\nonumber\\
\qquad \rightarrow \int^t_0S(t-s)\int_{\mathbf{X}}G(s,X^g_s,v)(g(s,v)-1)\nu(dv)ds.\label{gcon}
\end{align}
Due to the demicontinuity of $f$, we have for every fixed $h\in H$ and $s\in [0,T]$, as $n\rightarrow \infty$ 
$$\inp{S(t-s)\left(f(X^{g^n}_s)-f(X^g_s)\right),h}=\inp{f(X^{g^n}_s)-f(X^g_s),S^*(t-s)h}\rightarrow  0.$$
From the a priori estimate and linear growth condition on $f$ the sequence 
$\{S(t-s)\left(f(X^{g^n}_s)-f(X^g_s)\right), \; n\in\mathbb{N}\}$ is uniformly bounded. Therefore, using the dominated convergence theorem we get 
\begin{align}
&\inp{\int^t_0S(t-s)\left(f(X^{g^n}_s)-f(X^g_s)\right)ds,h}\nonumber\\
&=\int^t_0\inp{S(t-s)\left(f(X^{g^n}_s)-f(X^g_s)\right),h}ds\rightarrow 0.\nonumber
\end{align}
For the other term, similar to the estimates obtained for $I_2, I_3$, we have 
\begin{align}
\int^t_0 &\int_{\mathbf{X}}S(t-s)[G(s,X^{g^n}_s,v)(g^n-1)-G(s,X^g_s,v)(g-1)]\nu(dv)ds\nonumber\\
&= \int^t_0\int_{\mathbf{X}}S(t-s)[G(s,X^{g^n}_s,v)-G(s,X^g_s,v)](g^n-1)\nu(dv)ds\nonumber\\
&+\int^t_0 \int_{\mathbf{X}}S(t-s)G(s,X^g_s,v)[g^n-g]\nu(dv)ds\nonumber\\
&\leq\norm{X^{g^n}-X^g}_{\infty} \int^t_0  \int_{\mathbf{X}}S(t-s)\norm{G(s,v)-G(s,v)}_1(g^n-1)\nu(dv)ds\nonumber\\
&+(1+\norm{X^g}_{\infty})\int^t_0 \int_{\mathbf{X}}S(t-s)\norm{G(s,v)}_0[g^n-g]\nu(dv)ds.\nonumber
\end{align}
The convergence $\norm{X^{g^n}-X^g}_{\infty}\rightarrow 0$ and Remark \ref{gintrem} imply that the first integral converges  to zero. Similarly, Lemma \ref{hcon} and a priori estimate imply this assertion for the second integral. So, from the two convergences (\ref{fcon}), (\ref{gcon}) we can conclude that $X^g$ is a solution of Equation (\ref{maineq}). $\hspace*{3.7cm} \square $

\begin{rem}
The uniqueness of solution of Equation (\ref{maineq}) can easily be proved. But this property is not required in our arguments. 
\end{rem}


\subsection{\textbf{The Rate Function}}
In this section we prove that the rate function given by (\ref{rate}) is a good rate function. In the other words, we verify the condition $(A1)$ in Definition 1.1. In this way we have to deal with some technical difficulties. We apply the Yosida approximation to take advantage of strong solution properties and integration by part formulae. 
\begin{thm}\label{h2}
The level sets of rate function given by (\ref{rate}) are compact or equivalently the set 
$$K_N=\{\mathcal{G}^0(g):g\in \mathbf{S}^N\} $$
is a compact subset of $\mathcal{C}([0,T];H)$.
\end{thm}

\begin{proof}
It suffices to show that for an arbitrary sequence $\{g^n, n\in\mathbb{N}\}\subset S^N$,  $X^n=\mathcal{G}^0(g^n)$ has a convergence subsequent in the strong topology of 
$\mathcal{C}([0,T];H)$. We proved this convergence in previous section, when $\{g^n\}$ converges strongly. 

Since $S^N$ is weakly compact (cf. appendix in \cite{kn:BCD}), 
we can assume that $\{g^n\}$ has a subsequent, call it for simplicity in notation $\{g^n\}$, that  converges weakly to $g^0$. We apply the Yosida approximation to prove 
the desired convergence indirectly. 

Let $A_k=A(I-k^{-1}A)^{-1}, k\in \mathbb{N}$ be the Yosida approximation of $A$, which is a bounded operator. Then the following equation has the strong solution
\begin{align}
\frac{d}{dt}X^{n,k}_t=A_kX^{n,k}_t+f(X^{n,k}_t)+\int_{\mathbf{X}}G(t,X^{n,k}_t,v)(g(t,v)-1)\nu(dv),\nonumber
\end{align}
which can be represented in the form 
\begin{align}
X^{n,k}_t=S(t)X_0&+\int^t_0e^{(t-s)A_k}f(X^{n,k}_s)ds\nonumber\\
&+\int^t_0\int_{\mathbf{X}}e^{(t-s)A_k}G(s,X^{n,k}_s,v)(g^n(s,v)-1)
\nu(dv)ds.\nonumber
\end{align}
For a fixed $n$ we show that $\lim_{k\rightarrow \infty}\norm{X^{n,k}-X^n}_{\infty}=0$. From the definitions we have 
\begin{align}
&X^n_t-X^{n,k}_t\nonumber\\
\hspace{.2cm}&=\int^t_0\left(S(t-s)-e^{(t-s)A_k}\right)\left[f(X^n_s)+\int_{\mathbf{X}}G(s,X^n_s,v)(g^n-1)\nu(dv)\right] ds\nonumber\\
\hspace{0.3cm}&+\int^t_0e^{(t-s)A_k}\left[f(X^n_s)-f(X^{n,k}_s)+\int_{\mathbf{X}}[G(s,X^n_s,v)-G(s,X^{n,k}_s,v)](g^n-1)\nu(dv)\right] ds\nonumber\\
\hspace{0.3cm}&=:I^{n,k}_t+J^{n,k}_t.\label{first}
\end{align}
To estimate $\norm{I^{n,k}}_{\infty}$, we define 
\begin{align}
\lambda^k(s):=\sup_{t\in [s,T]}\norm{\left(S(t-s)-e^{(t-s)A_k}\right)[f(X^n_s)+\int_{\mathbf{X}}G(s,X^n_s,v)(g^n-1)\nu(dv)]}.\label{lambda}
\end{align}
We assumed without loss of generality that $S(t)$ is a contraction which implies that 
\begin{align}
\norm{S(t-s)-e^{(t-s)A_k}}_{L}\leq 2,\label{yos1}
\end{align}
where $\norm{\cdot}_{L}$ stands for the operator norm. Now, using the a priori estimate, linear growth condition on $f$ and Lemma \ref{gint}, we can conclude that    
$$\int^T_0\sup_{k\in\mathbb{N}}\lambda^k(s)ds<\infty.$$
So, if we prove that  $\lambda^k$ converges to zero pointwise, we can conclude, by using the Lebesgue dominated convergence theorem, that 
\begin{align}
\lim_{k\rightarrow \infty}\int^T_0\lambda^k(s)ds=0.\label{lambda}
\end{align}
Moreover, from the definition of $\lambda^k$ we obtain $\norm{I^{n,k}_t}\leq \int^t_0\lambda^k_sds$. Therefore the convergence (\ref{lambda}) would imply the desired convergence:    
\begin{align}
\lim_{k\rightarrow \infty}\norm{I^{n,k}}_{\infty}=0.\label{I}
\end{align}

From the Yosida approximation properties (cf. \cite{kn:KA}), we have for every $h\in H$
\begin{align}
\lim_{k\rightarrow \infty}(e^{tA_k}-S(t))h=0.\label{yos2}
\end{align}
We fix $s\in[0,T]$ and define the family of functions $\{\gamma^{k}, k\in\mathbb{N} \}$ on the interval $[s,T]$ by 
$$\gamma^{k}(t):=\left(S(t-s)-e^{(t-s)A_k}\right)[f(X^n_s)+\int_{\mathbf{X}}G(s,X^n_s,v)(g^n-1)\nu(dv)].$$
The inequality (\ref{yos1}) easily implies that the above family is equicontinuous. So, according to the Arzela-Ascoli theorem and (\ref{yos2}) we can conclude that for every $s\in [0,T]$, this family converges uniformly to zero on the interval $[s,T]$. But this is equivalent to the desired pointwise convergence 
$$\lim_{k\rightarrow\infty}\lambda^{k}(s)=0, \qquad \forall s\in [0,T],$$
that ascertains the convergence (\ref{I}).

To estimate the term $\norm{J^{n,k}_t}$, we use the energy type inequality and write 
\begin{align}
\norm{J^{n,k}_t}^2 &=\norm{-I^{n,k}_t+X^{n}_t-X^{n,k}_t}^2\nonumber\\
&\leq 2\int^t_0\inp{-I^{n,k}_t+X^n_s-X^{n,k}_s,f(X^n_s)-f(X^{n,k}_s)}ds\nonumber\\
&+2\int^t_0\inp{-I^{n,k}_t+X^n_s-X^{n,k}_s,\int_{\mathbf{X}}[G(s,X^n_s,v)-G(s,X^{n,k}_s,v)](g^n-1)\nu(ds)}ds\nonumber\\
&\leq 2\int^t_0\norm{I^{n,k}_s}\norm{f(X^n_s)-f(X^{n,k}_s)} ds\nonumber\\
&+\int^t_0\norm{I^{n,k}_s}\norm{X^n_s-X^{n,k}_s}\int_{\mathbf{X}} \norm{G(s,v)}_1(g^n(s,v)-1)\nu(dv)ds\nonumber\\
&+\int^t_0M\norm{X^n_s-X^{n,k}_s}^2ds+\int^t_0\norm{X^n_s-X^{n,k}_s}^2
\int_{\mathbf{X}}\norm{G(s,v)}_1(g^n-1)\nu(dv)ds.\label{J}
\end{align}
Due to the convergence (\ref{I}) the first and second terms converge to zero. So, regarding (\ref{I}) we can write 
$$\norm{X^n_t-X^{n,k}_t}^2\leq \delta+\int^t_0\left(M+g(s)\right)\int^t_0\norm{X^n_s-X^{n,k}_s}^2ds,$$
where $g(s)=\int_{\mathbf{X}}\norm{G(s,v)}_1(g^n(s,v)-1)\nu(dv)$ and $\delta>0$ is arbitrary. From Remark \ref{gintrem}, we have $\int^T_0g(s)ds<\infty$. Therefore the Gronwall's inequality yields for every $n\geq 0$
\begin{align}
\lim_{k\rightarrow\infty} \norm{X^n-X^{n,k}}_{\infty}=0.\label{fs}
\end{align}
\vspace{.5cm}

In the second step, we show the following uniform convergence 
\begin{align}
\lim_{n\rightarrow \infty}\sup_{k\in\mathbb{N}}\norm{X^{n,k}-X^{0,k}}_{\infty}=0.\label{ucon}
\end{align}
According to the assumptions, the semigroup $S(t)$ is a contraction which implies that $e^{tA_k}$ is a contraction, too. So, by energy type inequality we get

\begin{align}
&\norm{X^{n,k}_t-X^{0,k}_t}^2\nonumber\\
&\leq 2\int^t_0\inp{X^{n,k}_s-X^{0,k}_s,f(X^{n,k}_s)-f(X^{0,k}_s)}ds\nonumber\\
\hspace{0.5cm}&+ 2\int^t_0\inp{X^{n,k}_s-X^{0,k}_s,\int_{\mathbf{X}}[G(s,X^{n,k}_s,v)-G(s,X^{0,k}_s,v)](g^n-1)\nu(dv)}ds\nonumber\\
&+2\int^t_0\inp{X^{n,k}_s-X^{0,k}_s,\int_{\mathbf{X}}G(s,X^{0,k}_s,v)[g^n(s,v)-g^0(s,v)]\nu(dv)}ds\nonumber\\
&\leq M\int^t_0\norm{X^{n,k}_s-X^{0,k}_s}^2ds\nonumber\\
&+2\int^t_0\norm{X^{n,k}_s-X^{0,k}_s}^2\int_{\mathbf{X}}\norm{G(s,v)}_1|g^n(s,v)-1|\nu(dv)ds\nonumber\\
&+2\int^t_0\inp{X^{n,k}_s-X^{0,k}_s,\int_{\mathbf{X}}G(s,X^{0,k}_s,v)[g^n(s,v)-g^0(s,v)]\nu(dv)}ds.\label{step2}
\end{align}
It seems that the integration by part formula is required in estimating the last term. To this end, we define 

\begin{align}
h^n(t):=\int^t_0\int_{\mathbf{X}}G(s,X^{n}_s,v)\abs{g^n-g^0}\nu(dv)ds,\nonumber
\end{align}
and
\begin{align}
h^{n,k}(t):=\int^t_0\int_{\mathbf{X}}G(s,X^{n,k}_s,v)\abs{g^n-g^0}\nu(dv)ds.\nonumber
\end{align}
From Lemma \ref{hcon}, we have 
\begin{align}
\lim_{n\rightarrow\infty}&\int^T_0\int_{\mathbf{X}}\frac{\norm{G(s,X^{n}_s,v)}}{1+\norm{X^{n}_s}}\abs{g^n-g^0}\nu(dv)ds\nonumber\\
&\leq\lim_{n\rightarrow\infty}\int^t_0\int_{\mathbf{X}}\norm{G(s,v)}_0\abs{g^n-g^0}\nu(dv)ds=0.\nonumber
\end{align}
So, regarding the a priori estimate $\sup_n \norm{X^n}_{\infty}<\infty$,  we get
$$\lim_{n\rightarrow\infty}\int^T_0\int_{\mathbf{X}}\norm{G(s,X^{n}_s,v)}\abs{g^n-g^0}\nu(dv)ds=0.$$
This implies 
\begin{align}
\lim_{n\rightarrow\infty}\norm{h^{n}}_{\infty}=0.\label{hnormcon}
\end{align}
Furthermore, we have 
\begin{align}
h^{n,k}(t)-h^n(t)&=\int^t_0\int_{\mathbf{X}}[G(s,X^{n,k}_s,v)-G(s,X^{n}_s,v)]\abs{g^n-g^0}\nu(dv)ds\nonumber\\
&\leq\norm{X^{n,k}-X^{n}}_{\infty}\int^t_0\int_{\mathbf{X}}\norm{G(s,v)}_1\abs{g^n-g^0}\nu(dv)ds.\nonumber
\end{align}
Considering the proof of Lemma \ref{apriorilem}, we easily obtain  the uniform estimate  $$\sup_{n,k}\norm{X^{n,k}-X^{n}}_{\infty}<\infty.$$
This together with Lemma \ref{hcon} imply that
$$\lim_{n\rightarrow \infty}\sup_{k\in\mathbb{N}}\norm{h^{n,k}-h^n}_{\infty}=0.$$
So, we can conclude from (\ref{hnormcon}) that
\begin{align}
\lim_{n\rightarrow \infty}\sup_{k\in\mathbb{N}}\norm{h^{n,k}}_{\infty}=0.\label{hucon}
\end{align}
Now, using the integration by parts formula for the strong solution, we get  
\begin{align}
\int^t_0 &\inp{X^{n,k}_s-X^{0,k}_s,\int_{\mathbf{X}}G(s,X^{0,k}_s,v)(g^n(s,v)-g^0(s,v))\nu(dv)}ds\nonumber\\
&=\inp{X^{n,k}_t-X^{0,k}_t,h^{n,k}_t}-\int^t_0\inp{(X^{n,k}_s-X^{0,k}_s)^{\prime},h^{n,k}_s}ds\nonumber\\
&\leq \norm{X^{n,k}_t-X^{0,k}_t}\norm{h^{n,k}_t}-\int^t_0\inp{(X^{n,k}_s-X^{0,k}_s)^{\prime},h^{n,k}_s}ds.\label{bypart}
\end{align}
Since $X^{n,k}_s-X^{0,k}_s$ is a strong solution, we have
\begin{align}
\int^t_0 &\inp{(X^{n,k}_s-X^{0,k}_s)^{\prime},h^{n,k}_s}ds\nonumber\\
=&\int^t_0\inp{A_k(X^{n,k}_s-X^{0,k}_s)+f(X^{n,k}_s)-f(X^{0,k}_s),h^{n,k}_s}\nonumber\\
&+\int^t_0\inp{\int_{\mathbf{X}}G(s,X^{n,k}_s,v)(g^n(s,v)-1)\nu(dv),h^{n,k}_s}ds\nonumber\\
&-\int^t_0\inp{\int_{\mathbf{X}}G(s,X^{0,k}_s,v)(g^0(s,v)-1)\nu(dv),h^{n,k}_s}ds\nonumber\\
\leq & \norm{h^{n,k}}_{\infty}\int^t_0\{2\norm{X^{n,k}_s-X^{0,k}_s}+\norm{f(X^{n,k}_s)}+\norm{f(X^{0,k}_s)}ds\}\nonumber\\
&+\norm{h^{n,k}}_{\infty}\int^t_0(1+\norm{X^{n,k}_s})\int_{\mathbf{X}}\norm{G(s,v)}_0
\abs{g^n(s,v)-1}\nu(dv)ds\nonumber\\
&+\norm{h^{n,k}}_{\infty}\int^t_0(1+\norm{X^{0,k}_s})\int_{\mathbf{X}}\norm{G(s,v)}_0
\abs{g^0(s,v)-1}\nu(dv)ds.\nonumber
\end{align}
Due to the uniform convergence (\ref{hucon}), the uniform a priori estimate  and Lemma \ref{gint} the right hand side tends to zero uniformly for $k\in\mathbb{N}$, as $n\rightarrow\infty$. So, using the inequalities (\ref{step2}), (\ref{bypart}) and the Gronwal's lemma we can conclude the uniform convergence (\ref{ucon}).  

In view of the convergence (\ref{fs}), we have $\lim_{k\rightarrow\infty} \norm{X^0-X^{0,k}}_{\infty}=0$. Therefore from the convergences (\ref{fs}) and (\ref{ucon}) and by applying a  $3\eps$-argument we obtain the desired convergence 
$$\lim_{n\rightarrow\infty} \norm{X^n-X^0}_{\infty}=0.$$
\end{proof}



\subsection{\textbf{Weak Convergence of Solution}}
In this section we check the main assumption of Theorem \ref{mainthm} and prove that the map $\mathcal{G}^\eps$ is continuous w.r.t. the control function $g^{\eps}$. In our argument we need an a priori estimate for the mild solution of Equation (\ref{maineps}). Although this estimate can be concluded indirectly by \cite[Theorem 10]{kn:SZ}, We prove this estimate in the following proposition by a simple argument.
\begin{prop}
Let $\mathbb{E}\norm{X_0}^2<\infty$ and $\{g^{\eps}, \eps>0\}\subset S^N$ be a family of control functions. Consider $\mathcal{G}^{\eps}(\eps N^{\eps^{-1} g^{\eps}})=X^{\eps}$ as the mild solution of Equation (\ref{maineps}) with initial value $X_0$, or
\begin{align}
X^{\eps}_t=S(t-s)X_0&+\int^t_0S(t-s)f(X^{\eps}_s))ds\nonumber\\
&+\int^t_0\int_{\mathbf{X}}S(t-s)G(s,X^{\eps}_{s-},v)\left(\eps N^{\eps^{-1} g^{\eps}}(dvds)-\nu(dv)ds\right).\nonumber
\end{align}
Then we have the following a priori bound 
\begin{align}
\sup_{\eps>0}\mathbb{E}\left(\sup_{0\leq t\leq T}\norm{X^{\eps}_t}^2\right)<\infty. \label{aprioristoch}
\end{align}
\end{prop}

\begin{proof}
We write the It\^{o} type inequality for $X^{\eps}_t$ as
\begin{align}
\norm{X^{\eps}_t}^2 &\leq \norm{X_0}^2+2\int^t_0\inp{X^{\eps}_s,f(X^{\eps}_s)}ds+\int^t_0\int_{\mathbf{X}} \norm{\eps G(s,X^{\eps}_{s-},v)}^2 N^{\eps^{-1} g^{\eps}}(dvds)\nonumber\\
&\quad +2\int^t_0\int_{\mathbf{X}}\inp{ X^{\eps}_{s-},\eps G(s,X^{\eps}_{s-},v)}\left( N^{\eps^{-1} g^{\eps}}(dvds)-\nu(dv)ds\right).\nonumber
\end{align}
Hence by getting expectation, we have 
\begin{align}
\mathbb{E}\norm{X^{\eps}_t}^2 &\leq \mathbb{E}\norm{X_0}^2+C\int^t_0\mathbb{E}(1+\norm{X^{\epsilon}_s}^2)ds\nonumber\\
&\quad +\eps\mathbb{E}\int^t_0\int_{\mathbf{X}} \norm{G(s,X^{\eps}_{s},v)}^2 g^{\eps}\nu(dv)ds\nonumber\\
&\quad +2\mathbb{E}\int^t_0\int_{\mathbf{X}}\norm{X^{\eps}_s}\norm{G(s,X^{\eps}_{s},v)} (g^{\eps}-1)\nu(dv)ds.\label{ap}
\end{align}
For the third term we have 
\begin{align}
&\eps\mathbb{E}\int^t_0\int_{\mathbf{X}} \norm{G(s,X^{\eps}_{s},v)}^2 g^{\eps}\nu(dv)ds\nonumber\\
&\qquad \leq \eps\mathbb{E}\int^t_0(1+\norm{X^{\eps}_{s}})^2\int_{\mathbf{X}}\norm{G(s,v)}^2_0 g^{\eps}\nu(dv)ds.\label{aa}
\end{align}
Furthermore, for the last term in (\ref{ap}) we obtain
\begin{align}
&2\mathbb{E}\int^t_0\int_{\mathbf{X}}\norm{X^{\eps}_{s}}\norm{ G(s,X^{\eps}_{s},v)}(g^{\eps}-1)\nu(dv)ds\nonumber\\
&\quad \leq \mathbb{E}\int^t_0\int_{\mathbf{X}}\left( \norm{X^{\eps}_{s}}^2+\norm{ G(s,X^{\eps}_{s},v)}^2 \right)(g^{\eps}-1)\nu(dv)ds\nonumber\\
&\quad \leq \mathbb{E}\int^t_0\norm{X^{\eps}_{s}}^2\int_{\mathbf{X}}(g^{\eps}-1)\nu(dv)ds
+\mathbb{E}\int^t_0\int_{\mathbf{X}}\norm{ G(s,X^{\eps}_{s},v)}^2(g^{\eps}-1)\nu(dv)ds.\label{aaa}
\end{align}
The last integral can be estimated similar to (\ref{aa}). To estimate the first integral we use the assumption $L_T(g^{\epsilon})<N$ which yields $\int^t_0\int_{\mathbf{X}}(g^{\epsilon}-1)\nu(dv)ds<N$. So, using Lemma \ref{gint} and  Gronwall's lemma we can conclude the estimate (\ref{aprioristoch}). Note that our argument gives a uniform bound for all $\epsilon >0$. 
\end{proof}

To verify Hypothesis 1(ii), we must estimate $\norm{X^\eps-X^0}_\infty$ in distribution.
\begin{thm}\label{h3}
Let $\{g^\eps:\eps>0\} \subseteq \mathcal{U}^N$ for some
$N<\infty$. Moreover, assume that $g^\eps$ converges to $g^0$ in
distribution (as $S^N$-valued random variables), as $\eps\rightarrow
0$. Then
$$X^{\eps}=\mathcal{G}^\eps\para{\eps N^{\eps^{-1}g^{\eps}}}\rightarrow X^0=\mathcal{G}^0(g^0),$$
in distribution in the space $\mathcal{D}([0,T];H)$.
\begin{proof}
The proof is similar to that of Theorem $\ref{h2}$. The
differences are in some stochastic integrals for which we will
use Burkholder-type inequality to estimate them. Let, as before, $A_k$, $k\in \mathbb{N}$, be the Yosida approximation of the linear operator $A$. Since $A_k$ is bounded, the equation
$$dX^{\eps,k}_t=\para{A_kX^{\eps,k}_t+f(X^{\eps,k}_t)}dt+\int_{\mathbf{X}} G(t,X^{\eps,k}_t,v)\left(\eps N^{\eps^{-1}g^{\eps}}(dvdt)-\nu(dvdt)\right)$$
with initial value $X_0$ has a strong solution which can be represented in the form
\begin{align}
X^{\eps,k}_t=e^{(t-s)A_k}X_0&+\int^t_0e^{(t-s)A_k}f(X^{\eps,k}_s)ds\nonumber\\
&+\int^t_0\int_{\mathbf{X}} e^{(t-s)A_k}G(t,X^{\eps,k}_t,v)\left(\eps N^{\eps^{-1}g^{\eps}}(dvds)-\nu(dv)ds\right).\label{mildyosida}
\end{align}
By the definitions, we have
\begin{align}
&X^\eps_t-X^{\eps,k}_t=\left(S(t)-\emph{e}^{tA_k}\right) X_0\nonumber\\
&+\int^t_0\para{S(t-s)-\emph{e}^{(t-s)A_k}}f(X^\eps_s)ds\nonumber\\
&+\int^t_0\emph{e}^{(t-s)A_k}\para{f(X^\eps_s)-f(X^{\eps,k}_s)}ds\nonumber\\
&+\int^t_0\int_{\mathbf{X}}\para{S(t-s)-\emph{e}^{(t-s)A_k}} G(s,X^{\eps}_{s-},v)\left(\eps N^{\eps^{-1}g^{\eps}}(dvds)-\nu(dv)ds\right)\nonumber\\
&+\int^t_0\int_{\mathbf{X}}\emph{e}^{(t-s)A_k}\para{G(s,X^{\eps}_{s-},v)-G(s,X^{\eps,k}_s,v)}\left(\eps N^{\eps^{-1}g^{\eps}}(dvds)-\nu(dv)ds\right)\nonumber\\
&=:J^{k}_0+J^{\eps,k}_1(t)+J^{\eps,k}_2(t)+J^{\eps,k}_3(t)+J^{\eps,k}_4(t).\nonumber
\end{align}
From the Yosida approximation properties and square integrability of $X_0$, we get $\norm{J^{k}_0}\rightarrow 0$ in distribution. 
In estimating $J^{\eps,k}_3, J^{\eps,k}_4$ we apply the Burkholder-type inequality. So we have 
\begin{align}
\mathbb{E}\; \norm{J^{\eps,k}_3}^2_{\infty}&\leq \eps \mathbb{E}\int_{\mathbf{X}_T} \norm{G(t,X^\eps_t,v)}^2\abs{g^{\eps}-1}\nu(dv)dt\nonumber\\
&\leq \eps C \mathbb{E} (1+\norm{X^\eps}_{\infty})^2  \int_{\mathbf{X}_T}\norm{G(t,v)}^2_0\abs{g^{\eps}-1}\nu(dv)dt\nonumber\\
&\leq \eps C L_0\mathbb{E} (1+\norm{X^\eps}_{\infty})^2,\nonumber
\end{align}
in which $L_0=\sup_{g\in S^N}\int_{\mathbf{X}_T}\norm{G(t,v)}^2_0\abs{g-1}\nu(dv)dt$. 
Therefore, due to the a priori estimate (\ref{aprioristoch})  we can conclude that  $\norm{J^{\eps,k}_3}^2_{\infty}$ tends to zero in distribution, as $\eps\rightarrow 0$.

Similarly, we have for the  term $J^{\eps,k}_4$ 
\begin{align}
\mathbb{E} \norm{J^{\eps,k}_4}^2_{\infty}&\leq \eps C \mathbb{E}\int_{\mathbf{X}_T}\norm{G(t,X^{\eps}_t,v)-G(t,X^{\eps,k}_t,v)}^2\abs{g^{\eps}-1}\nu(dv)dt\nonumber\\
&\leq \eps C \mathbb{E}\left\{\norm{X^{\eps}-X^{\eps,k}}_{\infty}^2\int_{\mathbf{X}_T}\norm{G(t,v)-G(t,v)}^2_1\abs{g^{\eps}-1}\nu(dv)dt\right\}\nonumber\\
&\leq \eps CL_1\mathbb{E} \norm{X^{\eps}-X^{\eps,k}}_{\infty}^2,\nonumber
\end{align}
in which $L_1=\sup_{g\in S^N}\int_{\mathbf{X}_T}\norm{G(t,v)}^2_1\abs{g-1}\nu(dv)dt$. 
Again  the a priori estimate  implies that $\norm{J^{\eps,k}_4}^2_{\infty}$ tends to zero in distribution, as $\eps\rightarrow 0$.

By an easier argument than used in estimating
$\norm{I^{n,k}}_{\infty}$ and $\norm{J^{n,k}}_{\infty}$, in previous section, we can
conclude that
$\norm{J^{\eps,k}_1}_{\infty},\norm{J^{\eps,k}_2}_{\infty}\rightarrow 0$
in distribution, as $k\rightarrow\infty$. Hence, we have the
following convergence in distribution for any $\eps>0$
\begin{align}
\lim_{k\rightarrow\infty}\norm{X^{\eps}-X^{\eps,k}}_{\infty}=0.\label{36}
\end{align}

We now proceed in another direction and let $\eps\rightarrow 0$. Recall that 
\begin{align}
X^{0,k}_t=e^{(t-s)A_k}X_0 &+\int^t_0e^{(t-s)A_k}f(X^{0,k}_s)ds\nonumber\\
&\hspace{0.05cm}+\int^t_0\int_{\mathbf{X}} e^{(t-s)A_k}G(s,X^{0,k}_s,v)(g^0-1)\nu(dv)ds.\nonumber
\end{align}
It\^{o}'s formula for the strong solution yields 
\begin{align}
&\norm{X^{\eps,k}_t-X^{0,k}_t}^2 \nonumber\\
&\leq 2\int^t_0\inp{X^{\eps,k}_s-X^{0,k}_s,A_k\para{X^{\eps,k}_s-X^{0,k}_s}+f(X^{\eps,k}_s)-f(X^{0,k}_s)}ds\nonumber\\
&+2\int^t_0\int_{\mathbf{X}}\inp{X^{\eps,k}_s-X^{0,k}_s,G(s,X^{\eps,k}_s,v)-G(s,X^{0,k}_s,v)}\para{g^{\eps}-1}\nu(dv)ds\nonumber\\
&+2\int^t_0\int_{\mathbf{X}}\inp{X^{\eps,k}_s-X^{0,k}_s,G(s,X^{0,k}_s,v)}\para{g^{\eps}-g^0}\nu(dv)ds\nonumber\\
&+2\eps\int^t_0\int_{\mathbf{X}}\inp{X^{\eps,k}_{s_-}-X^{0,k}_{s_-},G(s,X^{\eps,k}_{s_-},v)}\para{N^{\eps^{-1}g^{\eps}}(dvds)-\eps^{-1} g^{\eps}\nu(dv)ds}\nonumber\\
&+\eps\int^t_0\int_{\mathbf{X}}\norm{G(s,X^{\eps,k}_s,v)}^2 g^{\eps}\nu(dv)ds \nonumber\\
&=:I^{\eps,k}_1(t)+I^{\eps,k}_2(t)+I^{\eps,k}_3(t)+I^{\eps,k}_4(t)+I^{\eps,k}_5(t).\label{12}
\end{align}

Since $A_k$ is negative definite, we have $I^{\eps,k}_1(t)\leq M\int^t_0\norm{X^{\eps,k}_s-X^{0,k}_s}^2ds$. For the term $I^{\eps,k}_2$, we write
\begin{align}
I^{\eps,k}_2(t)&\leq 2\int^t_0\norm{X^{\eps,k}_s-X^{0,k}_s}^2\int_{\mathbf{X}}\norm{G(s,v)}_1\para{g^{\eps}-1}\nu(dv)ds.\nonumber
\end{align}

We obtain the following estimate for the term $I^{\eps,k}_3$
\begin{align}
\mathbb{E} I^{\eps,k}_3(t)&\leq 2\mathbb{E} \int^t_0\int_{\mathbf{X}}\norm{X^{\eps,k}_s-X^{0,k}_s}(1+\norm{X^{0,k}_s})\norm{G(s,v)}_0\para{g^{\eps}-g^0}\nu(dv)ds\nonumber\\
&\leq 2\mathbb{E}\left\{\norm{X^{\eps,k}-X^{0,k}}_{\infty}(1+\norm{X^{0,k}}_{\infty})\int^t_0\int_{\mathbf{X}}\norm{G(s,v)}_0\para{g^{\eps}-g^0}\nu(dv)ds\right\}.\nonumber
\end{align}
From the proof of a priori estimate, we have $\sup_{k\in\mathbb{N}}\mathbb{E}\norm{X^{\eps,k}-X^{0,k}}_{\infty}(1+\norm{X^{0,k}}_{\infty})<\infty$. Moreover Lemma \ref{hcon} implies that the integral at the right hand side tends to zero in distribution. Therefore by a simple argument in probability theory we obtain the following convergence in distribution
$$\lim_{\eps\rightarrow 0}\sup_{k\in\mathbb{N}}\norm{I^{\eps,k}_3}_{\infty}=0.$$

Similarly, to estimate $I^{\eps,k}_4$ we write 
\begin{align}
I^{\eps,k}_4(t)&\leq 2\eps\int^t_0\int_{\mathbf{X}}\norm{X^{\eps,k}_{s_-}-X^{0,k}_{s_-}} \norm{G(s,X^{\eps,k}_s,v)}\para{N^{\eps^{-1}g^{\eps}}(dvds)-\eps^{-1} g^{\eps}\nu(dv)ds}.\nonumber
\end{align}
Applying Lemma \ref{lem1}, we obtain 
\begin{align}
\mathbb{E}\norm{I^{\eps,k}_4}_{\infty}\leq 3\eps\mathbb{E}\norm{X^{\eps,k}-X^{0,k}}^2_{\infty}+3\eps\mathbb{E}\int^T_0\int_{\mathbf{X}}(1+\norm{X^{\eps,k}}_{\infty})^2\norm{G(s,v)}^2_0 g^{\eps}\nu(dv)ds.\label{secondterm}
\end{align}
By a similar argument employed in estimating $J^{\eps,k}_3$, we can conclude that the second term tends to zero in distribution uniformly for $k\in\mathbb{N}$. clearly, due to the uniform a priori bound, we have this convergence for the first term, too. Therefore we have the following convergence in distribution
  $$\lim_{\eps\rightarrow 0}\sup_{k\in\mathbb{N}}\norm{I^{\eps,k}_4}_{\infty}=0.$$

Finally, the estimate obtained for the second term in (\ref{secondterm}) yields the corresponding convergence for the term $I^{\eps,k}_5$  

  $$\lim_{\eps\rightarrow 0}\sup_{k\in\mathbb{N}}\norm{I^{\eps,k}_5}_{\infty}=0.$$

Now regarding the estimates for $I^{\eps,1}, I^{\eps,2}$ and the convergences obtained for  $I^{\eps,k}_j, 3\leq j\leq 5$, we can apply Gronwall's lemma to $\norm{X^{\eps,k}_t-X^{0,k}_t}^2$ to conclude the following convergence in distribution
\begin{align}
\lim_{\eps\rightarrow 0}\sup\limits_{k\in \mathbb{N}}\norm{X^{\eps,k}-X^{0,k}}_\infty=0.\nonumber
\end{align}
This together with (\ref{36}) imply the desired convergence in distribution 
$$\lim_{\eps\rightarrow 0}\norm{X^{\eps}-X^0}_{\infty}=0.$$

\end{proof}
\end{thm}


\section{\textbf{Examples}}
\subsection{\textbf{Stochastic Heat Equation}}
As a simple example, we consider the stochastic heat equation which is the most typical equation studied in different approaches in SPDE. The LDP of this equation with Wiener noise and monotone drift has been studied in \cite{kn:DZ}.

Consider the following initial-boundary value problem where $D$ is a bounded domain with smooth boundary in $\mathbb{R}^d$.  

\begin{align}
\begin{array}{c}
  \frac{\partial u}{\partial t}=\Delta u+f(u(t,x)), \quad (t,x)\in [0,\infty)\times D,\label{25}\\
  u(t,x)=0, \quad \forall (t,x)\in [0,\infty)\times \partial D,\\
  u(0,x)=u_0(x) \quad \forall x\in D,
\end{array}
\end{align}
where, $u_0\in L^2(D)$.

Supposing $H=L^2(D)$ with the norm $\norm{\cdot}_{L^2(D)}$ and $A=\Delta u$ with the domain
$$D(A)=\left\{u\in L^2(D)\; :\; u', u''\in L^2(D), \; u(x)=0\; \forall x\in \partial D \right\},$$
the operator $A:D(A)\subset H\rightarrow H$ generates a
strongly continuous semigroup $S(t), t\geq 0$ on $H$. Let $\mathbf{X}=\Real$ and assume that the functions $G_i:[0,\infty)\times H\times\mathbf{X}\rightarrow \Real, 1\leq i\leq l$ satisfy Hypothesis 3. Furthermore, define the diffusion coefficient $g_i:[0,\infty)\times H\rightarrow\mathbb{R}$ as 
$$g_i(t,u):=\int_{\mathbf{X}}G_i(t,u,v)\tilde{N}_i(dv), \qquad i=1,\ldots,l$$
in which $\tilde{N}_i, 1\leq i\leq l$ are independent Poisson random measures.

We assume the following conditions on the diffusion coefficient $g_i$ and drift coefficient  $f:H\rightarrow \Real.$\\

\textbf{Hypothesis 4}
\begin{description}
  \item (i) There exist a function $a\in L^2(D)$ and a constant $C>0$ such that 
$$\norm{f(u)}\leq a+C\norm{u}, \quad \forall \; u\in H;$$
  \item (ii) $f$ is demicontinuous and $-f$ is uniformly semimonotone with parameter M, i.e.
$$\inp{f(y)-f(x),y-x}\leq M \norm{y-x}, \quad \forall x, y\in H.$$
  \item (iii) $g_i(t,.)$, for each $0\leq i\leq l,$ is uniformly Lipschitz with Lipschitz constant $M>0$, i.e.
$$\norm{g_i(t,y)-g_i(t,x)}\leq M\norm{y-x}, \quad \forall x, y\in H;$$
\end{description}

The stochastic evolution Equation  with L\'{e}vy noise corresponding to Equation (\ref{25}) can be written as
\begin{align}
du_t=Au_t dt+f(u_t)dt+\sum^l_{i=1}\int_{\mathbf{X}}G_i(t,u_t,v)\tilde{N}_i(dvdt), \quad u(0)=u_0.\label{heatlevy}
\end{align}
Now, according to Theorem \ref{mainthm} we obtain the LDP for the mild solution of Equation (\ref{heatlevy}).


\subsection{\textbf{Semilinear Hyperbolic Systems}}
Studying the hyperbolic systems of the following type is a particular feature of the semigroup approach. 
\begin{align}
\begin{array}{c}
  \frac{\partial u}{\partial t}=\sum^n_{i=1} a_i(x)\frac{\partial u}{\partial x_i}+b(x)u+f(u), \\
  u(0,x)=u_0, \quad u_0(x)\in L^2(\Real^n;\Real^k), \quad x\in \Real^n,\label{26}
\end{array}
\end{align}
where $u=(u_1,\ldots,u_k)^t$ is the unknown vector, $f:H\rightarrow \Real^k$ and the coefficients  $a_i(x), b(x)$ are square matrices of order $k$, for every $i=1,\ldots,n$ and $x\in \Real^n$. Let $H=L^2(\Real^n;\Real^k),\; \mathbf{X}=\Real^m$ be the state space of solution and noise, respectively. Furthermore, we define the diffusion coefficient $g:[0,\infty)\times H\rightarrow \Real^k $ similar to the diffusion $g_i$ defined in previous example. Also, we assume assumptions similar to Hypothesis 4 on $f, g$ and the following conditions on $a_i, b$.\\

\textbf{Hypothesis 5}
\begin{description}
  \item (i) The matrices $a_i(x)$, $i=1,\ldots,n$ and $x\in \Real^n$, are symmetric;
  \item (ii) The components of $a_i$, $i=1,\ldots,n$, their first order derivatives, and the function $b$ are bounded continuous;
\end{description}

We define the closed unbounded operator $A$ on $H$ by
$$Au=\sum^n_{i=1} a_i(x)\frac{\partial u}{\partial x_i}+b(x)u, \quad \forall u\in D(A)\subset H.$$
According to \cite[Theorem 3.51]{kn:Ta}, $A$ is the generator of a
$C_0$-semigroup on $H$. So, the Equation $(\ref{26})$ together with  L\'{e}vy noise,  can be represented as the following stochastic semilinear evolution equation
\begin{align}
du_t=Au_tdt+f(u_t)dt+\int_{\mathbf{X}}G(t,u_t,v)\tilde{N}(dvdt), \quad u(0)=u_0.\label{Hypermain}
\end{align}

Now, due to Hypotheses 4 and 5, it is clear that the conditions of Hypothesis 2 are being fulfilled. 
Therefore, we claim that the mild solution of Equation (\ref{Hypermain}) satisfies LDP.

\subsection{\textbf{SPDE's Arising in HJM Models of Interest Rate}}
The main assumption in modelling the interest rate is its fluctuations around a deterministic mean value. Therefore,  large deviation principle can be an interesting issue for those models. A great achievement in modelling interest rate is the HJM model, in which the forward rates are being calibrated and expressed by the mild solution of an SPDE. The forward rate $f_t(x)$ at time $t\geq 0$ for time to maturity $x\geq 0$ is given by
$$f_t(x)=-\frac{\partial}{\partial x}\log(P(t,t+x)),$$
where $P(t,t+x)$ is the price at time $t\geq 0$ of a zero-coupon bond with maturity $T=t+x$.

Empirical studies have shown that term structure models based on Brownian motion provide a poor fit to market data. Bj\"{o}rk et al. \cite{kn:BDKR} and Eberlein \cite{kn:ER} proposed a process with jump instead of the Wiener process. 

Considering the L\'{e}vy noise, we can write the forward rate under the risk neutral measure as the mild solution of the following SPDE  
\begin{align}
df_t(x)=\para{\frac{\partial}{\partial x}f_t(x)+a_t(x)}dt+\sum^{\infty}_{i=1}\sigma^i_t(x) dZ^i_t,\label{hjmspde}
\end{align}
where $Z^i_t, i\in\mathbb{N}$ are L\'{e}vy processes on $\mathbb{R}$ and the drift $a$ is derived from $\sigma$ by the HJM no-arbitrage condition. See \cite{kn:JZ} for a more general setting with  infinite dimensional L\'{e}vy noise. To express the HJM no-arbitrage condition, suppose the cumulant generating functions  
$$\Psi^i(u)=\ln \mathbb{E}[exp(uZ^i_1)], \qquad i\in\mathbb{N},$$
exist on the intervals $[a_i, b_i]$ containing zero as an inner point.
Then the drift $a$ is defined by 
$$a_t(x)=-\sum^{\infty}_{i=1}\sigma^i_t(x)\Psi^{\prime}\para{-\int^x_0\sigma^i_t(s)ds}.$$
Now we can write the mild solution of (\ref{hjmspde}) as 
$$f_t(x)=S(t)f_0(x)+\int^t_0S(t-s)a_s(x)ds+\sum^{\infty}_{i=1}\int^t_0S(t-s)\sigma^i_s(x)dX^i_s, $$
where the semigroup $S(t)$ is the shift operator $S(t)f(x)=f(x+t)$ and is generated by the differential operator $A f(x)=\frac{\partial}{\partial x}f(x)$. 

It is well known that the Lipschitz condition on diffusion $\sigma$ implies the Lipschitz condition on the drift $a$, ( cf. \cite{kn:FT}). So, we can assume that the coefficients of Equation (\ref{hjmspde}) satisfy Hypothesis 2. This implies that the HJM model of interest rate with multiplicative L\'{e}vy noise possesses the large deviation property. 

We should notice that we can consider a more general diffusion coefficient that varies on the state space $\mathbf{X}$, which in that case it must satisfies Hypothesis 3, too.

\subsection{Stochastic Delay Differential  Equations}
Peszat and Zabczyk \cite{kn:PZ} have studied the existence and uniqueness of the following stochastic delay differential equation with  Lipschitz coefficients. Moreover, the equation with  monotone drift coefficient has been studied  in \cite{kn:SZ}. 
\begin{align}
dX(t)=&\left(\int^0_h X(t+\theta)\mu (d\theta)\right)dt+f(X(t))dt+g(X(t))dZ_t\nonumber\\
&\qquad X(\theta)=\phi(\theta), \quad \theta\in (-h,0].\label{delayeq}
\end{align}
where $h>0$, $\mu$ is a finite variation measure on $(-h,0]$, $Z$ is a square integrable L\'{e}vy process in $\mathbb{R}^m$. Moreover  $f:\mathbb{R}^n\rightarrow \mathbb{R}^n$ and $g:\mathbb{R}^n\rightarrow L(\mathbb{R}^m, \mathbb{R}^n)$ are monotone and Lipschitz  maps, respectively. We explore the solution in the space $H=\left(L^2((-h,0);\mathbb{R}^n), \mathbb{R}^n \right)$. 

Consider the operator 
$$A(u, v)=(\int^0_{-h} u(\xi)\mu(d\xi), u(0)),$$
where the domain of $A$ is equal to 
$$D(A)=\{(u, v): u\in W^{2,2}((-h,0);\mathbb{R}^n), v=u(0)\}.$$
Defining $F(u, v)=(f(v), 0)$ and $G(u, v)=(g(v),0)$ and assuming that $f$ has linear growth, we can easily check that $F$ and $G$ satisfy Hypothesis 2(ii, iv, v).

Therefore, letting $X(t)=(x_t, x(t))\in H$, where $x_t\in L^2((-h,0);\mathbb{R}^n)$ and  $$x_t(\theta)=x(t+\theta),\qquad \theta\in (-h,0),$$ 
we can write Equation (\ref{delayeq}) as 
\begin{align}
dX_t=&[AX_t+F(X_t)]dt+\int_{\mathbf{X}}G(X_t,v)\tilde{N}(dvdt),\nonumber\\
&\qquad X(0)=(\phi, \phi(0)).\nonumber
\end{align}
This equation coincides  with the framework of Equation \ref{mainlevyeq}. So, we can claim the LDP for stochastic delay differential equation with L\'{e}vy noise.



\emph{Department of Mathemats}\\
\emph{Institute for Advanced Studies in Basic Sciences}\\
\emph{P.O. Box 45195-1159}\\
\emph{Zanjan, Iran}\\
\emph{e-mail: dadashi@iasbs.ac.ir}\\

\end{document}